\documentclass[10pt]{amsart}

\title{Upper Triangularity for Unipotent Representations}

\author{Lucas Mason-Brown}

\usepackage{amsmath,amssymb,amsthm,amsfonts,graphics ,tikz-cd, mdframed,enumitem,appendix, scalerel,stackengine,float}
\usepackage{hyperref}

\theoremstyle{plain}
\newtheorem{theorem}{Theorem}[section]

\newtheorem{example}{Example}[section]
\newtheorem{proposition}{Proposition}[section]
\newtheorem{definition}{Definition}[section]
\newtheorem{question}{Question}[section]

\newcommand{\gr}{\operatorname{gr}}

\newcommand{\AV}{\operatorname{AV}}
\newcommand{\ad}{\operatorname{ad}}
\newcommand{\Ad}{\operatorname{Ad}}
\newcommand{\Unip}{\operatorname{Unip}}
\newcommand{\unip}{\operatorname{unip}}
\newcommand{\Deg}{\operatorname{Deg}}
\newcommand{\Aut}{\operatorname{Aut}}
\newcommand{\Prim}{\operatorname{Prim}}

\begin{document}

\begin{abstract}
Suppose $G$ is a real reductive group. The determination of the irreducible unitary representations of $G$ is one of the major unsolved problem in representation theory. There is evidence to suggest that every irreducible unitary representation of $G$ can be constructed through a sequence of well-understood operations from a finite set of building blocks, called the unipotent representations. These representations are `attached' (in a certain mysterious sense) to the nilpotent orbits of $G$ on the dual space of its Lie algebra. Inside this finite set is a still smaller set, consisting of the unipotent representations attached to non-induced nilpotent orbits. In this paper, we prove that in many cases this smaller set generates (through a suitable kind of induction) all unipotent representations. 
\end{abstract}

\maketitle

Suppose $G$ is a real reductive group. There are three powerful techniques for producing irreducible unitary representations of $G$: parabolic induction, cohomological induction, and the formation of complementary series. These techniques produce most, but not all, irreducible unitary representations. If $G = SL_2(\mathbb{R})$, there are exactly three missing representations: the two limit of discrete series representations and the trivial representation. These missing representations are called \emph{unipotent}. They are linked in a mysterious way to the nilpotent orbits of $G$ on the dual space of its Lie algebra. In the case of $SL_2(\mathbb{R})$, there are exactly three such orbits: two principal nilpotent orbits and $\{0\}$. The limit of descrete series representations correspond (in a precise sense) to the principal nilpotent orbits, and the trivial representation corresponds to $\{0\}$. 

The search for a general theory of unipotent representations has a long and colorful history. Highlights include character formulas, due to Barbasch and Vogan (\cite{BarbaschVogan1985}), for unipotent representations of complex reductive groups and a construction, due to Torasso (\cite{Torasso1997}), of the unipotent representations attached to minimal nilpotent orbits. Despite these results (and others), the general theory remains elusive. In this article, we propose a blueprint. 

If $\mathbf{G}$ is a complex reductive algebraic group, and $\mathbf{L} \subset \mathbf{G}$ is a Levi subgroup, there is a correspondence
$$\mathrm{Ind}^{\mathfrak{g}}_{\mathfrak{l}}: \{\text{nilpotent co-adjoint }\mathbf{L}-\text{orbits}\} \to \{\text{nilpotent co-adjoint }\mathbf{G}-\text{orbits}\}$$
called the induction of nilpotent orbits. We propose the following general strategy for understanding the unipotent representations of $G$:

\begin{enumerate}
    \item Understand the unipotents $\Unip(\mathcal{O}_{\mathfrak{g}})$ attached to a non-induced $\mathbf{G}$-orbit $\mathcal{O}_{\mathfrak{g}}$
    \item Find a recipe for building the unipotents $\Unip(\mathcal{O}_{\mathfrak{g}})$ attached to an induced $\mathbf{G}$-orbit $\mathcal{O}_{\mathfrak{g}} = \mathrm{Ind}^{\mathfrak{g}}_{\mathfrak{l}}\mathcal{O}_{\mathfrak{l}}$ from the unipotents $\Unip(\mathcal{O}_{\mathfrak{l}})$ attached to $\mathcal{O}_{\mathfrak{l}}$
\end{enumerate}

In a previous paper (\cite{MasonBrown2018}), we have made some encouraging progress towards $(1)$. In this paper, we will turn our attention towards $(2)$. Our main result in this direction is that the set $\Unip(\mathcal{O}_{\mathfrak{g}})$ of unipotent representations attached to $\mathcal{O}_{\mathfrak{g}} = \mathrm{Ind}^{\mathfrak{g}}_{\mathfrak{l}}\mathcal{O}_{\mathfrak{l}}$ is related in the K-theory of finite-length $(\mathfrak{g},\mathbf{K})$-modules by an upper triangular matrix to a set of standard classes related to $\Unip(\mathcal{O}_{\mathfrak{l}})$. This result is strong evidence in favor of the strategy outlined above: the unipotent representations attached to non-induced orbits are indeed the building blocks from which all others can be formed. There is some hope that the relationship between $\Unip(\mathcal{O}_{\mathfrak{l}})$ and $\Unip(\mathcal{O}_{\mathfrak{g}})$ can be made extremely precise in certain special cases (in \cite{MasonBrown2019}, this is achieved for the principal nilpotent orbit).

In Section \ref{sec:preliminaries}, we will review some preliminary facts and constructions. In Section \ref{sec:unipotentrepresentations}, we will provide an operational definition of unipotent representations. In Section \ref{sec:uppertriangularity} we will prove our main result (Theorem \ref{thm:unipotentsanddegenerates}). This will require brief digressions into the theory of twisted $\mathcal{D}$-modules and translation functors (which we define in a rather general setting). In Section \ref{sec:examples}, we apply our main result to the special case of the $2^n$-orbit of $Sp(2n,\mathbb{R})$ (for several small values of $n$).

\section{preliminaries}\label{sec:preliminaries}

\subsection{Harish-Chandra Isomorphism}

Let $\mathfrak{g}$ be a complex reductive Lie algebra and let $\mathfrak{h}$ be a Cartan subalgebra of $\mathfrak{g}$. Write $\Delta(\mathfrak{g},\mathfrak{h}) \subset \mathfrak{h}^*$ for the corresponding root system, $\Delta(\mathfrak{g},\mathfrak{h})^{\vee} = \{\alpha^{\vee}: \alpha \in \Delta(\mathfrak{g},\mathfrak{h})\}$ for the co-roots, and $W(\mathfrak{g},\mathfrak{h}) \subset GL(\mathfrak{h}^*)$ for the Weyl group. Since $W(\mathfrak{g},\mathfrak{h})$ is independent (up to canonical isomorphism) of $\mathfrak{h}$, we will sometimes write simply $W(\mathfrak{g})$. 

Recall that the center $Z(\mathfrak{g})$ of the universal enveloping algebra $U(\mathfrak{g})$ of $\mathfrak{g}$ is identified by the Harish-Chandra isomorphism
$$\zeta: Z(\mathfrak{g}) \cong \mathbb{C}[\mathfrak{h}^*]^{W(\mathfrak{g})}$$
with the algebra of $W(\mathfrak{g})$-invariant polynomial functions on $\mathfrak{h}^*$. If $\lambda \in \mathfrak{h}^*$, there is a character $\gamma_{\lambda}$ of $Z(\mathfrak{g})$ defined by
$$\gamma_{\lambda}(X) := \zeta(X)(\lambda) \in \mathbb{C}$$
Since $\zeta$ is an isomorphism, every character of $Z(\mathfrak{g})$ arises in this fashion and $\gamma_{\lambda} = \gamma_{\mu}$ if and only if $\lambda \in W(\mathfrak{g})\mu$. 

\subsection{$(\mathfrak{g},\mathbf{K})$-Modules}

Let $\mathfrak{g}$ be a (finite-dimensional) complex Lie algebra and let $\mathbf{K}$ be a complex algebraic group. We say that $(\mathfrak{g},\mathbf{K})$ is a \emph{pair} if
\begin{enumerate}
    \item The Lie algebra $\mathfrak{k}$ of $\mathbf{K}$ is a subalgebra of $\mathfrak{g}$
    \item $\mathbf{K}$ acts on $\mathfrak{g}$ by Lie algebra automorphisms $\Ad(k) \in \mathrm{Aut}(\mathfrak{g})$ extending the adjoint action of $\mathbf{K}$
    \item The Lie algebra of $\Ad(\mathbf{K})$ is the subalgebra $\ad(\mathfrak{k}) \subset \ad(\mathfrak{g})$
\end{enumerate}

\begin{definition}\label{def:gkmodule}
Let $(\mathfrak{g},\mathbf{K})$ be a pair. A $(\mathfrak{g}, K)$-module is a complex vector space $V$ with a Lie algebra action of $\mathfrak{g}$ and an algebraic group action of $\mathbf{K}$ such that
\begin{enumerate}
\item\label{cond:ugk1} The action map
\begin{equation}
\mathfrak{g} \otimes_{\mathbb{C}} V \to V
\end{equation}
is $\mathbf{K}$-equivariant, and
\item\label{cond:ugk2} The $\mathfrak{g}$-action, restricted to the subspace $\mathfrak{k} \subset \mathfrak{g}$, coincides with the differentiated action of $\mathbf{K}$.
\end{enumerate}
A morphism of $(\mathfrak{g},\mathbf{K})$-modules is a linear map which commutes with the actions of $\mathfrak{g}$ and $\mathbf{K}$. We will write $M(\mathfrak{g},\mathbf{K})$ the abelian category of $(\mathfrak{g},\mathbf{K})$-modules. An object $M \in M(\mathfrak{g},\mathbf{K})$ has finite-length if it has a finite composition series
$$0 = M_0 \subset M_1 \subset ... \subset M_n = M$$
with irreducible quotients. We will write $M^{\mathrm{fl}}(\mathfrak{g},\mathbf{K})$ for the full subcategory of finite-length $(\mathfrak{g},\mathbf{K})$-modules and $KM^{\mathrm{fl}}(\mathfrak{g},\mathbf{K})$ for its Grothendieck group.
\end{definition}

\subsection{Dixmier Algebras}

Let $\mathbf{G}$ be a complex algebraic group and let $\mathfrak{g}$ be its Lie algebra.

\begin{definition}\label{def:dixmier}
A Dixmier algebra for $\mathbf{G}$ is a triple $(B,\Ad, \phi)$ consisting of a complex algebra $B$, an algebraic group action $\Ad: \mathbf{G} \to \Aut(B)$, and an algebra homomorphism
$$\phi: U(\mathfrak{g}) \to B$$
satisfying three properties:

\begin{enumerate}
    \item $\phi$ commutes with the adjoint actions of $\mathbf{G}$
    \item The differential of the action $\Ad$ of $\mathbf{G}$ on $B$ (denoted $\ad$) coincides with the difference of the left and right $\mathfrak{g}$-actions defined by $\phi$:
    $$\ad(X)b = \phi(X)b - b\phi(X) \qquad X \in \mathfrak{g}, b \in B$$
    \item\label{cond:dixmier3} $B$ is a finite-length $U(\mathfrak{g})$-bimodule
\end{enumerate}
Almost always, $\Ad$ and $\phi$ will be clear from the context. When this is the case, we will omit them from our notation. 
\end{definition}

The prototypical example of a Dixmier algebra is the quotient of $U(\mathfrak{g})$ by a primitive ideal. More generally,

\begin{proposition}[\cite{Jantzen1983}, Kapitel 6]\label{prop:dixmier1}
Let $I \subset U(\mathfrak{g})$ be a two-sided ideal such that $I \cap Z(\mathfrak{g})$ has finite-codimension in $Z(\mathfrak{g})$. Then $A=U(\mathfrak{g})/I$ is a Dixmier algebra.
\end{proposition}

Now let $B$ be a Dixmier algebra and let $(F,\rho)$ be a finite-dimensional representation of $\mathbf{G}$. We will define the structure of a Dixmier algebra on $B \otimes \mathrm{End}(F)$. 

The $\mathbf{G}$-action on $F$ lifts to a Lie algebra map $\rho: \mathfrak{g} \to \mathrm{End}(F)$, which lifts to an algebra homomorphism $\rho: U(\mathfrak{g}) \to \mathrm{End}(F)$. Let
$$\Delta: U(\mathfrak{g}) \to U(\mathfrak{g}) \otimes U(\mathfrak{g}) \qquad \Delta(X) = X \otimes 1 + 1 \otimes X$$
be the co-multiplication on $U(\mathfrak{g})$. $\Delta$ is an algebra homomorphism. Consider the composite
$$(\varphi \otimes \rho) \circ \Delta: U(\mathfrak{g}) \to B \otimes \mathrm{End}(F)$$
This map will play the role of $\phi$ for the algebra $B \otimes \mathrm{End}(F)$. The action of $\mathbf{G}$ on $B \otimes \mathrm{End}(F)$ is the diagonal one. 

\begin{proposition}[\cite{Jantzen1983},Kapitel 6]\label{prop:dixmier2}
Let $B$ be a Dixmier algebra and let $F$ be a finite-dimensional representation of $\mathbf{G}$. Then $B \otimes \mathrm{End}(F)$ is a Dixmier algebra (with $\phi$ and $\Ad$ defined as above).
\end{proposition}

In analogy with Definition \ref{def:gkmodule}, we define

\begin{definition}\label{def:dixmiermodule}
Let $\mathbf{K} \subset \mathbf{G}$ be an algebraic subgroup, and let $B$ be a Dixmier algebra for $\mathbf{G}$. A $(B,\mathbf{K})$-module is a (left) $B$-module $V$ with an algebraic action of $\mathbf{K}$ such that

\begin{enumerate}
    \item The action map
    $$B \otimes V \to V$$
    Is $\mathbf{K}$-equivariant, and
    \item The $\mathfrak{k}$-action on $V$ coming from the $B$-action on $V$ and the homomorphism $\phi: \mathfrak{k} \subset U(\mathfrak{g}) \to B$ coincides with the differentiated action of $\mathbf{K}$.
\end{enumerate}

A morphism of $(B,\mathbf{K})$-modules is a $\mathbf{K}$-equivariant $B$-module homomorphism. We will write $M(B,\mathbf{K})$ for the abelian category of $(B,\mathbf{K})$-modules and $M^{\mathrm{fl}}(B,\mathbf{K})$ for the full subcategory of finite-length modules.
\end{definition}





\subsection{Parabolic Induction}\label{sec:parabolicinduction}

Let $G$ be a real reductive group (for us this will mean the real points of a connected reductive algebraic group defined over $\mathbb{R}$). Fix a maximal compact subgroup $K \subset G$ and write $\theta: G \to G$ for the associated Cartan involution. Write $\mathbf{G}$ and $\mathbf{K}$ for the complexifications (of $G$ and of $K$) and $\mathfrak{g}$ and $\mathfrak{k}$ for their Lie algebras. Then $(\mathfrak{g},\mathbf{K})$ is a pair. If $\mathfrak{q} \subset \mathfrak{g}$ is a parabolic subalgebra and $\mathbf{Q} \subset \mathbf{G}$ is the corresponding (connected) subgroup of $\mathbf{G}$, then $(\mathfrak{q},\mathbf{Q} \cap \mathbf{K}) \subset (\mathfrak{g},\mathbf{K})$ is a subpair. We will define a left-exact functor
\begin{equation}\label{eqn:zuckermanfunctor}
    \mathbf{I}^{(\mathfrak{g},\mathbf{K})}_{(\mathfrak{q},\mathbf{Q}\cap \mathbf{K})}: M(\mathfrak{h},\mathbf{L}) \to M(\mathfrak{g},\mathbf{K})
\end{equation}
This functor reduces to the usual functor of parabolic induction (when $\mathfrak{q}$ is germane), cohomological induction (when $\mathfrak{q}$ is $\theta$-stable), and real parabolic induction (when $\mathfrak{q}$ is the complexification of a real parabolic subalgebra).

Roughly speaking, if $W \in M(\mathfrak{l},\mathbf{Q} \cap \mathbf{K})$, $\mathbf{I}^{(\mathfrak{g},\mathbf{K})}_{(\mathfrak{l},\mathbf{Q}\cap \mathbf{K})}W$ is the $(\mathfrak{g},\mathbf{K})$-module
$$\mathbf{K}-\text{finite vectors in } \mathrm{Hom}_{\mathfrak{q}}(U(\mathfrak{g}),W \otimes \det(\mathfrak{u}))$$
where $\det(\mathfrak{u})$ is the top exterior power of $\mathfrak{u}$, the nilradical of $\mathfrak{q}$. This definition is not quite correct (or meaningful, strictly speaking) if $\mathbf{K}$ is disconnected. We review the correct definition below.

Let $\mathbf{K}^0$ denote the identity component of $\mathbf{K}$, and let $\mathbf{K}^1 = \mathbf{L}\mathbf{K}^0$. Then $\mathbf{I}^{(\mathfrak{g},\mathbf{K})}_{(\mathfrak{q},\mathbf{Q}\cap\mathbf{K})}W$ is defined in stages

\begin{enumerate}
    \item Let 
    $$W' := W \otimes \det(\mathfrak{u})$$
    \item Form
    $$\mathrm{Hom}_{\mathfrak{q}}(U(\mathfrak{g}),W')$$
    This vector space has the structure of a $(\mathfrak{g}, \mathbf{Q} \cap \mathbf{K})$-module
    \item Take $\mathbf{K}^0$-finite vectors
    $$\Gamma^0\mathrm{Hom}_{\mathfrak{q}}(U(\mathfrak{g}),W') := \mathrm{Hom}_{\mathfrak{q}}(U(\mathfrak{g}),W')_{\mathbf{K}^0}$$
    This vector space has the structure of a $\mathfrak{g}$-module with algebraic actions of $\mathbf{Q} \cap \mathbf{K}$ and $\mathbf{K}^0$. These group actions are both compatible with $\mathfrak{g}$, but not necessarily with eachother. They restrict to two (often distinct) actions of $\mathbf{Q} \cap \mathbf{K}^0$.
    \item Form the subspace of $\Gamma^0\mathrm{Hom}_{\mathfrak{q}}(U(\mathfrak{g}),W')$ on which both $\mathbf{Q} \cap \mathbf{K}^0$-actions coincide:
    $$\Gamma^1\mathrm{Hom}_{\mathfrak{q}}(U(\mathfrak{g}),W') := \{v \in \Gamma^0\mathrm{Hom}_{\mathfrak{q}}(U(\mathfrak{g}),W'): g \cdot_1 v = g \cdot_2 v \quad \forall g \in \mathbf{Q} \cap \mathbf{K}^0\}$$
    This is a $\mathfrak{g}$-module with compatible algebraic actions of $\mathbf{K}^0$ and $\mathbf{Q} \cap \mathbf{K}$, and hence of $\mathbf{K}^1 = (\mathbf{Q} \cap \mathbf{K})\mathbf{K}^0$
    \item Perform a finite induction
    $$\Gamma \mathrm{Hom}_{\mathfrak{q}}(U(\mathfrak{g}),W'):= \mathrm{Ind}^{\mathbf{K}}_{\mathbf{K}^1} \Gamma^1\mathrm{Hom}_{\mathfrak{q}}(U(\mathfrak{g}),W')$$
    This vector space has the structure of a $(\mathfrak{g},\mathbf{K})$-module.
\end{enumerate}

The assignment
$$W \mapsto \Gamma \mathrm{Hom}_{\mathfrak{q}}(U(\mathfrak{g}),W')$$
defines a left-exact functor, since each of its consituents (described in $(1)-(5)$ above) are exact or left-exact. Since the category $M(\mathfrak{g},\mathbf{K})$ has enough injectives, we can define the right derived functors
$$R^i\mathbf{I}^{(\mathfrak{g},\mathbf{K})}_{(\mathfrak{q},\mathbf{Q}\cap \mathbf{K})}: M(\mathfrak{q},\mathbf{Q}\cap \mathbf{K}) \to M(\mathfrak{g},\mathbf{K})$$
If we fix a Levi decomposition $\mathfrak{q} = \mathfrak{l} \oplus \mathfrak{u}$ (and hence a Levi decomposition $\mathbf{Q} = \mathbf{L} \mathbf{U}$), then $(\mathfrak{l},\frac{\mathbf{Q} \cap \mathbf{K}}{\mathbf{U} \cap \mathbf{K}})$ is a pair (with $\mathbf{Q} \cap \mathbf{K}$ acting on $\mathfrak{l}$ via the isomorphism $\mathfrak{l} \cong \mathfrak{q}/\mathfrak{u}$), and there is a surjective morphism of pairs
$$(\mathfrak{q}, \mathbf{Q} \cap \mathbf{K}) \twoheadrightarrow (\mathfrak{l},\frac{\mathbf{Q} \cap \mathbf{K}}{\mathbf{U} \cap \mathbf{K}})$$
Pulling back along this morphism defines a fully faithful embedding
$$M(\mathfrak{l},\frac{\mathbf{Q} \cap \mathbf{K}}{\mathbf{U} \cap \mathbf{K}}) \subset M(\mathfrak{q}, \mathbf{Q} \cap \mathbf{K})$$
The restriction of $\mathbf{I}^{(\mathfrak{g},\mathbf{K})}_{(\mathfrak{q},\mathbf{Q}\cap \mathbf{K})}$ to this subcategory has many favorable properties, which we summarize below.

\begin{theorem}\label{thm:propertiesofI}
Let $W \in M(\mathfrak{l}, \frac{\mathbf{Q}\cap \mathbf{K}}{\mathbf{U} \cap \mathbf{K}})$ and let $\mathfrak{h} \subset \mathfrak{l}$ be a Cartan subalgebra. Then
\begin{enumerate}
    \item If $W$ has finite-length, then $\mathbf{I}^{(\mathfrak{g},\mathbf{K})}_{(\mathfrak{q},\mathbf{Q}\cap \mathbf{K})}W$ has finite-length, for every $i \geq 0$
    \item If $W$ has infinitesimal character $\gamma_{\lambda}$, then $\mathbf{I}^{(\mathfrak{g},\mathbf{K})}_{(\mathfrak{q},\mathbf{Q}\cap \mathbf{K})}W$ has infinitesimal character $\gamma_{\lambda+\rho(\mathfrak{u})}$, for every $i \geq 0$
    \item There is an integer $s$ (depending only on $\mathfrak{q}$, $\mathbf{K}$, and $\mathfrak{g}$) such that $\mathbf{I}^{(\mathfrak{g},\mathbf{K})}_{(\mathfrak{q},\mathbf{Q}\cap \mathbf{K})}W = 0$ for every $i > s$
\end{enumerate}
\end{theorem}

For proofs, we direct the reader to \cite{Vogan1981}, Section 6.3. Vogan proves these statements under the assumption that $\mathfrak{q}$ is germane, but his arguments can be (slightly and easily) modified to accomodate our more general setting. 

In light of Theorem \ref{thm:propertiesofI}, we can define a group homomorphism
$$I(\mathfrak{l},\mathfrak{q},\cdot): KM^{\mathrm{fl}}(\mathfrak{l}, \frac{\mathbf{Q}\cap \mathbf{K}}{\mathbf{U} \cap \mathbf{K}}) \to KM^{\mathrm{fl}}(\mathfrak{g},\mathbf{K}) \qquad I(\mathfrak{l},\mathfrak{q},[W]) = \sum_i (-1)^i [R^i \mathbf{I}^{(\mathfrak{g},\mathbf{K})}_{(\mathfrak{q},\mathbf{Q}\cap \mathbf{K})}W]$$

This homomorphism will be our primary tool for constructing unipotent representations.

\subsection{Associated Varieties}

Let $\mathbf{G}$ be a complex reductive algebraic group and let $\mathfrak{g}$ be its Lie algebra. Write $\mathcal{N} \subset \mathfrak{g}$ for the $\mathbf{G}$ and $\mathbb{C}^{\times}$-invariant subset of nilpotent elements of $\mathfrak{g}$. $\mathbf{G}$ acts on $\mathcal{N}$ with finitely many orbits, which are called the \emph{nilpotent orbits of } $\mathbf{G}$.

If we fix a (non-degenerate, symmetric, bilinear $\mathbf{G}$-invariant) form $B(- , -)$ on $\mathfrak{g}$ (as we can, since $\mathfrak{g}$ is reductive), we get a $\mathbf{G}$-invariant isomorphism $\mathfrak{g} \cong \mathfrak{g}^*$. Let $\mathcal{N}^*$ be the image of $\mathcal{N}$ under this isomorphism. This set is independent of $B$---its elements are precisely the functionals $\lambda \in \mathfrak{g}^*$ which annihilate their centralizers in $\mathfrak{g}$. 

Now suppose $I \subset U(\mathfrak{g})$ is a two-sided ideal. If we equip $U(\mathfrak{g})$ with its usual filtration, then there is a canonical $\mathbf{G}$-invariant isomorphism (of graded commutative algebras) $\gr U(\mathfrak{g}) \cong S(\mathfrak{g})$. There is also an obvious isomoprhism $S(\mathfrak{g}) \cong \mathbb{C}[\mathfrak{g}^*]$. Hence, $\gr(I)$ is identified with a graded ideal in $\mathbb{C}[\mathfrak{g}^*]$.

Make the following

\begin{definition}
The \emph{associated variety} of $I$ is the Zariski-closed subset of $\mathfrak{g}^*$ defined by the graded ideal $\gr(I) \subset \mathbb{C}[\mathfrak{g}^*]$
\begin{align*}
\AV(I) &:= V(\gr(I))\\
&\subset \mathfrak{g}^*
\end{align*}
Since $\gr(I) \subset \mathbb{C}[\mathfrak{g}^*]$ is $\mathbf{G}$-invariant and graded, $\AV(I) \subset \mathfrak{g}^*$ is $\mathbf{G}$ and $\mathbb{C}^{\times}$-invariant. 
\end{definition}

\begin{theorem}[\cite{Vogan1991}, Theorem 5.7]\label{thm:finitecodimensionAV}
Suppose $I \cap Z(\mathfrak{g}) \subset Z(\mathfrak{g})$ is an ideal of finite codimension. Then $\AV(I) \subset \mathcal{N}^*$. 
\end{theorem}

In the setting of Theorem \ref{thm:finitecodimensionAV}, $\AV(I)$ is a finite union of $\mathbf{G}$-orbits on $\mathcal{N}^*$. If we write $\mathcal{O}_1,...,\mathcal{O}_n$ for the open $\mathbf{G}$-orbits on $\AV(I)$, then $\overline{\mathcal{O}}_i$ are the irreducible components of $\AV(I)$ and
$$\AV(I) = \bigcup \overline{\mathcal{O}}_i$$
When $I$ is primitive, $\AV(I)$ is irreducible.

\begin{theorem}[Joseph, Borho-Brylinski, \cite{Joseph1985}, \cite{BorhoBrylinski1985}]\label{thm:Josephirreducibility}
If $I$ is primitive, there is a $\mathbf{G}$-orbit $\mathcal{O} \subset \mathcal{N}^*$ such that
$$\AV(I) = \overline{\mathcal{O}}$$
\end{theorem}

\section{Unipotent Representations}\label{sec:unipotentrepresentations}

Let $\mathbf{G}$ be a complex connected reductive algebraic group and let $\mathfrak{g}$ be its Lie algebra. Fix a Cartan subalgebra $\mathfrak{h} \subset \mathfrak{g}$ and write $\Lambda \subset \mathfrak{h}^*$ for the lattice of integral weights.

If $I \subset U(\mathfrak{g})$ is a primitive ideal, then by Schur's lemma the intersection $I \cap Z(\mathfrak{g})$ is a maximal ideal in $Z(\mathfrak{g})$ and hence the kernel of an infinitesimal character $\gamma_{\lambda}$ for some $\lambda \in \mathfrak{h}^*$ well-defined up to $W(\mathfrak{g})$. If $\lambda \in \mathfrak{h}^*$, write $\Prim^{\lambda}U(\mathfrak{g})$ for the set of primitive ideals in $U(\mathfrak{g})$ of infinitesimal character $\gamma_{\lambda}$. Primitive ideals in universal enveloping algebras have been studied extensively by Dixmier, Duflo, Joseph, Vogan, Barbasch, and others. We will need only the following handful of results (proofs can be found in \cite{Dixmier1974}).

\begin{theorem}[\cite{Dixmier1974}]\label{thm:primitiveideals}
Let $\lambda \in \mathfrak{h}^*$. Then
\begin{enumerate}
    \item $\Prim^{\lambda}U(\mathfrak{g})$ is a finite set containing a unique minimal and unique maximal element.
    \item If $I, J \in \Prim^{\lambda}U(\mathfrak{g})$ and there is a strict inclusion $I \subset J$, then there is a strict inclusion $\AV(J) \subset \AV(I)$.
    \item  If $I \subset U(\mathfrak{g})$ is any prime two-sided ideal of infinitesimal character $\lambda$, then $I \in \mathrm{Prim}^{\lambda}U(\mathfrak{g})$.
\end{enumerate}
\end{theorem}

In \cite{BarbaschVogan1985}, Barbasch and Vogan use the following definition:

\begin{definition}\label{def:unipotent}
Let $\lambda \in \mathbb{Q} \otimes_{\mathbb{Z}} \Lambda \subset \mathfrak{h}^*$. A primitive ideal $I \in \Prim^{\lambda}U(\mathfrak{g})$ is unipotent if the following conditions are satisfied

\begin{enumerate}
    \item If $\mu \in \lambda + \Lambda$ and $J \in \Prim_{\mu}U(\mathfrak{g})$ satisfies 
    $$\AV(J) \subseteq \AV(I)$$
    then
    \begin{equation}\label{eqn:inequalityofinfchar}
        B(\mu, \mu) \geq B(\lambda,\lambda)
    \end{equation}
    \item If, in addition, \ref{eqn:inequalityofinfchar} is an equality, then $\mu \in W(\mathfrak{g})\lambda$ and
    $$\AV(J) = \AV(I)$$
\end{enumerate}
\end{definition}

In the following proposition, we catalog some of the elementary properties of unipotent ideals:

\begin{proposition}\label{prop:propsofunipotentideals}
For every $\lambda \in \mathfrak{h}^*$, there is at most one unipotent ideal in $\Prim^{\lambda}U(\mathfrak{g})$. If $I \in \Prim^{\lambda}U(\mathfrak{g})$ is unipotent, then $I$ is a maximal ideal and for every $J \in \Prim^{\lambda}U(\mathfrak{g})$, $J \neq I$, there is a proper inclusion
$$\AV(I) \subset \AV(J)$$
\end{proposition}

\begin{proof}
Choose a maximal ideal $J \subset U(\mathfrak{g})$ with $I \subseteq J$. Note that $J$ is primitive (in any associative algebra, every maximal ideal is primitive). The inclusion $I \subseteq J$ induces an inclusion $\AV(J) \subseteq \AV(I)$, which must be an equality by the second part of Definition \ref{def:unipotent}. By part 2 of Theorem \ref{thm:primitiveideals}, this implies $I=J$. Hence, $I$ is a maximal ideal. By part 1 of Theorem \ref{thm:primitiveideals}, $\Prim^{\lambda}U(\mathfrak{g})$ contains a unique maximal element. In particular, if $J \in \Prim^{\lambda}U(\mathfrak{g})$, $J \neq I$, then there is a strict inclusion $J \subset I$ and hence a strict inclusion $\AV(I) \subset \AV(J)$, as desired.
\end{proof}

We will say that an infinitesimal character $\gamma_{\lambda}$ is \emph{unipotent} if $\Prim^{\lambda}U(\mathfrak{g})$ contains a unipotent ideal $I_{\lambda} \subset U(\mathfrak{g})$. In this case, $I_{\lambda}$ is unique (by Proposition \ref{prop:propsofunipotentideals}) and is in fact the (unique) maximal ideal of infinitesimal character $\gamma_{\lambda}$. By Theorem \ref{thm:Josephirreducibility}, there is a (unique) $\mathbf{G}$-orbit $\mathcal{O} \subset \mathcal{N}^*$ such that
$$\AV(I_{\lambda}) = \overline{\mathcal{O}}$$
This defines a map from unipotent infinitesimal characters to $\mathbf{G}$-orbits on $\mathcal{N}^*$. Write $\unip(\mathcal{O})$ for the fiber of this map over a $\mathbf{G}$-orbit $\mathcal{O} \subset \mathcal{N}^*$. Usually, $\unip(\mathcal{O})$ will be infinite.

If we have in mind a pair $(\mathfrak{g},\mathbf{K})$, we will write $\Unip^{\lambda}(\mathcal{O})$  for the set of (isomorphism classes of) irreducible $(\mathfrak{g},\mathbf{K})$-modules $X$ with $\mathrm{Ann}(X) = I_{\lambda}$, and
$$\Unip(\mathcal{O}) := \bigsqcup_{\lambda \in \unip(\mathcal{O})} \Unip^{\lambda}(\mathcal{O})$$

\subsection{Special Unipotent Representations}

In \cite{BarbaschVogan1985}, Barbasch and Vogan study a distinguished class of unipotent representations related to the Arthur Conjectures (\cite{Arthur1983},\cite{Arthur1989}) called (by Barbasch and Vogan) the \emph{special} unipotent representations. For completeness, we will review their definition below.

In the setting of Section \ref{sec:unipotentrepresentations}, write $\mathbf{G}^{\vee}$ for the Langlands dual of $\mathbf{G}$ and $\mathfrak{g}^{\vee}$ for its Lie algebra. By construction, $\mathfrak{g}^{\vee}$ contains a distinguished Cartan subalgebra $\mathfrak{h}^{\vee}$ which is naturally identified with $\mathfrak{h}^*$. There is an order-reversing map
$$\psi: \mathcal{N}^{\vee}/\mathbf{G}^{\vee} \to  \mathcal{N}/\mathbf{G}$$
first defined by Spaltenstein (\cite{Spaltenstein1982}). If $\mathcal{O}^{\vee} \subset \mathcal{N}^{\vee}$ is a nilpotent $\mathbf{G}^{\vee}$-orbit, we can find a Lie algebra homomorphism (far from unique)
$$\phi_{\mathcal{O}^{\vee}}: \mathfrak{sl}_2(\mathbb{C}) \to \mathfrak{g}^{\vee} \qquad \phi(E) \in \mathcal{O}^{\vee}$$
Then $d^{\vee} := \frac{1}{2}\phi_{\mathcal{O}^{\vee}}(D)$ is a semisimple element of $\mathfrak{g}^{\vee}$. Conjugating by $\mathbf{G}^{\vee}$ if necessary, we can assume that $d^{\vee} \in \mathfrak{h}^{\vee} \cong \mathfrak{h}^*$. This element of $\mathfrak{h}^*$ is well-defined modulo $W(\mathfrak{g})$ and therefore determines a character
$$\gamma_{\mathcal{O}^{\vee}} := \gamma_{d^{\vee}}: Z(\mathfrak{g}) \to \mathbb{C}$$
by the Harish-Chandra isomorphism. If $\mathcal{O} \subset \mathcal{N}$ is a $\mathbf{G}$-orbit, let
$$\mathrm{arth}(\mathcal{O}) := \{\gamma_{\mathcal{O}^{\vee}}: \psi(\mathcal{O}^{\vee}) = \mathcal{O}\}$$
This is a finite (sometimes empty) set of small (often singular) infinitesimal characters associated to $\mathcal{O}$. If $\mathcal{O}'$ is a $\mathbf{G}$-orbit on $\mathcal{N}^*$, the set $\mathrm{arth}(\mathcal{O}')$ is defined by first replacing $\mathcal{O}'$ with its image in $\mathcal{N}$ under the $\mathbf{G}$-invariant identification $B:\mathfrak{g} \cong \mathfrak{g}^*$. Barbasch and Vogan prove

\begin{theorem}[\cite{BarbaschVogan1985}]
For every $\mathbf{G}$-orbit $\mathcal{O} \subset \mathcal{N}^*$, there is an inclusion
$$\mathrm{arth}(\mathcal{O}) \subseteq \unip(\mathcal{O})$$
\end{theorem}

If we have a pair $(\mathfrak{g},\mathbf{K})$ in mind, we will write
$$\Unip^{\mathrm{arth}}(\mathcal{O}) := \bigsqcup_{\lambda \in \mathrm{arth}(\mathcal{O})} \Unip^{\lambda}(\mathcal{O})$$

\begin{example}
If $G = SL_2(\mathbb{R})$, there are two $\mathbf{G}$-orbits on $\mathcal{N}^*$: $\{0\}$ and the principal nilpotent orbit, $\mathcal{O}^{\mathrm{prin}}$. Infinitesimal characters are parameterized by (weakly) dominant weights $\lambda \in \mathfrak{h}^*$. In standard coordinates
$$\unip(\{0\}) = \{1\} \qquad \unip(\mathcal{O}^{\mathrm{prin}}) = [0,\frac{1}{2}] \cap \mathbb{Q}$$
and
$$\mathrm{arth}(\{0\}) = \{1\} \qquad \mathrm{arth}(\mathcal{O}^{\mathrm{prin}}) = \{0\}$$
\end{example}

\section{Upper Triangularity for Unipotent Representations}\label{sec:uppertriangularity}

Let $\mathfrak{q}$ be a parabolic subalgebra of $\mathfrak{g}$. Fix a Levi decomposition $\mathfrak{q} = \mathfrak{l}\oplus \mathfrak{u}$ and a nilpotent $\mathbf{L}$-orbit $\mathcal{O}_{\mathfrak{l}} \subset \mathcal{N}_{\mathfrak{l}}$. Lusztig and Spaltenstein observed (\cite{LusztigSpaltenstein1979}) that the set $\mathbf{G} \cdot \left(\overline{\mathcal{O}}_{\mathfrak{l}} + \mathfrak{u}\right) \subset \mathcal{N}_{\mathfrak{g}}$ contains a unique open $\mathbf{G}$-orbit $\mathcal{O}_{\mathfrak{g}} \subset \mathcal{N}_{\mathfrak{g}}$, which (they proved) depends only on $\mathfrak{l}$ and $\mathcal{O}_{\mathfrak{l}}$. The correspondence $\mathcal{O}_{\mathfrak{l}} \mapsto \mathcal{O}_{\mathfrak{g}}$ is called the \emph{induction of nilpotent orbits} and is denoted by 
$$\mathcal{O}_{\mathfrak{g}} = \mathrm{Ind}^{\mathfrak{g}}_{\mathfrak{l}}\mathcal{O}_{\mathfrak{l}}$$
There is a parallel notion of induction for $(\mathfrak{g},\mathbf{K})$-modules. In Section \ref{sec:parabolicinduction}, we defined the category $M^{\mathrm{fl}}(\mathfrak{l},\frac{\mathbf{Q}\cap \mathbf{K}}{\mathbf{U}\cap \mathbf{K}})$ and the group homomorphism
$$I(\mathfrak{l},\mathfrak{q},\cdot): KM^{\mathrm{fl}}(\mathfrak{l},\frac{\mathbf{Q}\cap \mathbf{K}}{\mathbf{U}\cap \mathbf{K}}) \to KM^{\mathrm{fl}}(\mathfrak{g},\mathbf{K})$$
In Section \ref{sec:unipotentrepresentations}, we defined sets $\Unip(\mathcal{O}_{\mathfrak{l}})$ and $\Unip(\mathcal{O}_{\mathfrak{g}})$ consisting of irreducible $(\mathfrak{l},\frac{\mathbf{Q}\cap \mathbf{K}}{\mathbf{U}\cap \mathbf{K}})$-modules and $(\mathfrak{g},\mathbf{K})$-modules, respectively. The following question arises immediately

\begin{question}\label{question:induction}
What is the relationship between $\Unip(\mathcal{O}_{\mathfrak{g}})$ and $I(\mathfrak{l},\mathfrak{q},\Unip(\mathcal{O}_{\mathfrak{l}}))$?
\end{question}

A bit of thought reveals that this is not the right question to ask. The orbit $\mathcal{O}_{\mathfrak{g}}$ can be obtained by induction from a whole range of parabolics, of which $\mathfrak{q}$ is only one. Each of these parabolics contributes its own set of induced representations, which should be taken into account. To fix (and ultimately answer) Question \ref{question:induction}, we will consider the partial flag variety $\mathcal{Q}$ of parabolic subalgebras, $\mathbf{G}$-conjugate to $\mathfrak{q}$. If $\mathfrak{q}' \in \mathcal{Q}$, there is a group element $g \in \mathbf{G}$ such that $\Ad(g)\mathfrak{q} = \mathfrak{q}'$. If we fix a Levi decomposition $\mathfrak{q}' = \mathfrak{l}'\oplus \mathfrak{u}'$ and require that $\Ad(g)\mathfrak{l}=\mathfrak{l}'$, then $g$ is unique up to right multiplication by $\mathbf{L}$ and left multiplication by $\mathbf{L}'$. In particular, $\Ad(g)$ induces a \emph{canonical} isomorphism $Z(\mathfrak{l}) \cong Z(\mathfrak{l}')$. If $\gamma$ is a character of $Z(\mathfrak{l})$, we will write $\gamma'$ for the corresponding character of $Z(\mathfrak{l}')$.

\begin{definition}\label{def:degenerate}
Let $\gamma \in \unip(\mathcal{O}_{\mathfrak{l}})$ and denote by $X^{\vee}$ the contragradient of a $(\mathfrak{g},\mathbf{K})$-module $X$ (or its class). A degenerate representation for the data $(\mathfrak{q},\gamma)$ is a class in $KM^{\mathrm{fl}}(\mathfrak{g},\mathbf{K})$ of the form
$$I(\mathfrak{l}',\mathfrak{q}',W^{\vee})^{\vee}$$
where $\mathfrak{q}' \in \mathcal{Q}$, $\mathfrak{l}' \subset \mathfrak{q}'$ is a Levi subalgebra, and $W$ is an irreducible $(\mathfrak{l}',\frac{\mathbf{Q}'\cap \mathbf{K}}{\mathbf{U}'\cap \mathbf{K}})$-module with annihilator equal to the unipotent ideal $I_{\mu'} \subset U(\mathfrak{l}')$ (cf Section \ref{sec:unipotentrepresentations}). Write $\Deg^{\gamma}(\mathfrak{q})$ for the set of all such representations. 
\end{definition}

Fix, once and for all, a Cartan subalgebra $\mathfrak{h} \subset \mathfrak{l}$ and a positive system $\Delta^+(\mathfrak{g},\mathfrak{h})$ compatible with $\mathfrak{q}$. Suppose the infinitesimal character $\gamma$ in Definition \ref{def:degenerate} is represented by a functional $\mu \in \mathfrak{h}^*$. Then a degenerate representation $[M] \in \Deg^{\gamma_{\mu}}(\mathfrak{q})$ has infinitesimal character $\gamma_{\mu - \rho(\mathfrak{u})}$. Suppose that $\gamma_{\mu-\rho(\mathfrak{u})} \in \unip(\mathcal{O}_{\mathfrak{g}})$. A more reasonable formulation of Question \ref{question:induction} is the following

\begin{question}
What is the relationship between $\Unip^{\mu-\rho(\mathfrak{u})}(\mathcal{O}_{\mathfrak{g}})$ and $\Deg^{\mu}(\mathfrak{q})$? 
\end{question}

Our main result provides an answer to this question under some conditions on $\mu$ and on $\mathcal{O}_{\mathfrak{g}}$:

\begin{theorem}\label{thm:unipotentsanddegenerates}
Let $\mu \in \mathfrak{h}^*$. Suppose
\begin{enumerate}
    \item $\gamma_{\mu} \in \unip(\mathcal{O}_{\mathfrak{l}})$
    \item $\gamma_{\mu - \rho(\mathfrak{u})} \in \unip(\mathcal{O}_{\mathfrak{g}})$
    \item $\mu - \rho(\mathfrak{u})$ is antidominant for $\mathfrak{u}$
    \item The moment map
    $$\eta: \mathbf{G} \times_{\mathbf{Q}} \left(\overline{\mathcal{O}}_{\mathfrak{l}} + \mathfrak{u}\right) \to \overline{\mathcal{O}}_{\mathfrak{g}}$$
    is birational
\end{enumerate}
Then the sets $\Unip^{\mu-\rho(\mathfrak{u})}(\mathcal{O}_{\mathfrak{g}})$ and $\Deg^{\mu}(\mathfrak{q})$ have the same $\mathbb{Z}$-span in $KM^{\mathrm{fl}}(\mathfrak{g},\mathbf{K})$ and are related therein by an upper triangular matrix with $\pm 1$'s along the diagonal.
\end{theorem}

\subsection{$\mathcal{D}$-modules on $\mathcal{Q}$}\label{sec:Dmodules}

We begin by recalling some basic facts from the theory of algebraic $\mathcal{D}$-modules. For details and proofs, see \cite{BeilinsonBernstein1981} or \cite{HechtMilicicSchmidWolf}. Let $X$ be a smooth complex algebraic variety. Let $\mathcal{O}_X$ be its structure sheaf and let $\mathcal{D}_X$ be its sheaf of (local, algebraic) differential operators. A twisted sheaf of differential operators (a TDO, for short) is any sheaf of algebras on $X$ which is locally isomorphic to $\mathcal{D}$. If $\mathcal{D}'_X$ is a TDO, we can consider the abelian category $\mathcal{M}(\mathcal{D}'_X)$ of coherent (left) $\mathcal{D}'_X$-modules. Every object $\mathcal{M} \in \mathcal{M}(\mathcal{D}'_X)$ has a good filtration (by quasi-coherent subsheaves) and $\gr(\mathcal{M})$ is a coherent sheaf on the cotangent bundle $T^*X$. The support of this sheaf is a closed, conical subset $\mathrm{CV}(\mathcal{M}) \subset T^*X$ called the characteristic variety of $\mathcal{M}$. It is independent of the filtration used to define it. If we denote the support of $\mathcal{M}$ (in the usual sense) by $\mathrm{Supp}(\mathcal{M})$, then $\mathrm{Supp}(\mathcal{M})$ is the image of $\mathrm{CV}(\mathcal{M})$ under the projection map $T^*X \to X$.

If the characteristic variety of $\mathcal{M}$ has the same dimension as $X$ (i.e. half the dimension of $T^*X$), we say that $\mathcal{M}$ is holonomic. Holonomic $\mathcal{D}'_X$-modules have finite length. They form a subcategory of $\mathcal{M}(\mathcal{D}_X')$ which we will denote by $\mathcal{M}^{\mathrm{hol}}(\mathcal{D}'_X)$. If it is the sheaf of sections of a finite-rank vector bundle and its characteristic variety is the zero section of $T^*X$. In particular, $\mathcal{M}$ is holonimic. 

Now suppose $f: Y \hookrightarrow X$ is an open or closed embedding of smooth complex algebraic varieties. The restriction $f^*\mathcal{D}'_X$ of $\mathcal{D}'_X$ to $Y$ (in the sense of quasi-coherent sheaves) is a TDO, which we will denote by $\mathcal{D}'_Y$. There are two pullbacks
$$f^*,f^!: \mathcal{M}^{\mathrm{hol}}(\mathcal{D}'_X) \to \mathcal{M}^{\mathrm{hol}}(\mathcal{D}'_Y) $$
and two pushforwards
$$f_*,f_!: \mathcal{M}^{\mathrm{hol}}(\mathcal{D}'_Y)  \to \mathcal{M}^{\mathrm{hol}}(\mathcal{D}'_X)$$
In general, $f_*$ is right adjoint to $f^*$ and $f_!$ is left adjoint to $f^!$. In particular, the functors $f_*$ and $f^!$ are left exact and the functors $f^*$ and $f_!$ are right exact. If $f$ is an open embedding, then
\begin{enumerate}
    \item $f^*=f^!$ is the usual (exact) restriction functor (of quasi-coherent sheaves).
    \item $f_*$ is the usual pushforward (of quasi-coherent sheaves)
\end{enumerate}
If $f$ is a closed embedding, then 

\begin{enumerate}[resume]
    \item $f_*=f_!$ is an (exact) embedding of categories.
    \item $f^!$ is the usual restriction functor (of quasi-coherent sheaves)
    \item If $\mathcal{M}_Y(\mathcal{D}'_X)$ is the subcategory of $\mathcal{M}(\mathcal{D}'_X)$ consisting of objects supported in $Y$, then $f_*=f_!$ defines an equivalence $\mathcal{M}(\mathcal{D}'_Y) \cong \mathcal{M}(\mathcal{D}'_X)$ with inverse given by $f^!$. This fact, due to Kashiwara (\cite{Kashiwara1983}) is somewhat deeper than the rest. 
\end{enumerate}

If $\mathbf{K}$ is an algebraic group acting (algebraically) on $X$, then $\mathcal{D}$ is a sheaf of complex algebras with $\mathbf{K}$-action. A $\mathbf{K}$-equivariant TDO is a sheaf of algebras with $\mathbf{K}$-action which is locally isomorphic to $\mathcal{D}$. If $\mathcal{D}'_X$ is such a sheaf, we can form the abelian category $\mathcal{M}(\mathcal{D}'_X,\mathbf{K})$ of $\mathbf{K}$-equivariant coherent $\mathcal{D}'_X$-modules. If $\mathbf{K}$ acts on $X$ with finitely many orbits, then objects in this category are automatically holonomic (and hence automatically finite-length). The four functors defined above have $\mathbf{K}$-equivariant analogs which satisfy ($\mathbf{K}$-equivariant analogs of) properties $(1)-(5)$.

Return now to the setting of Theorem \ref{thm:unipotentsanddegenerates}. The $\mathbf{K}$-equivariant TDOs on $\mathcal{Q}$ are parameterized by one-dimensional representations of $\mathfrak{l}$. We will write $\mathcal{D}_{\mathcal{Q}}^{\lambda}$ for the TDO corresponding to a representation $\lambda$. If $\lambda$ is the trivial representation, then $\mathcal{D}_{\mathcal{Q}}^{\lambda} = \mathcal{D}_{\mathcal{Q}}$. If $\lambda$ integrates to a character of $\mathbf{L}$, then $\mathcal{D}_{\mathcal{Q}}^{\lambda} = \mathcal{L}^{\lambda} \otimes \mathcal{D}_{\mathcal{Q}} \otimes (\mathcal{L}^{\lambda})^{-1}$, where $\mathcal{L}^{\lambda}$ is (the sheaf of sections of) the $\mathbf{G}$-equivariant line bundle with fiber $\lambda$ over $\mathfrak{q} \in \mathcal{Q}$. For a definition of $\mathcal{D}_{\mathcal{Q}}^{\lambda}$ in the general case, see e.g. \cite{HechtMilicicSchmidWolf}. 

For any $\lambda$, $\Gamma(\mathcal{Q},\mathcal{D}_{\mathcal{Q}}^{\lambda})$ acquires the structure of an algebraic $\mathbf{G}$-module. An element $X 
\in \mathfrak{g}$ determines a right-invariant vector field $\xi_X \in \Gamma(\mathcal{D}_{\mathcal{Q}}^{\lambda},\mathcal{Q})$, and the correspondence $X \mapsto \xi_X$ induces a $\mathbf{G}$-equivariant map
$$\phi: U(\mathfrak{g}) \to \Gamma(\mathcal{Q},\mathcal{D}_{\mathcal{Q}}^{\lambda})$$
which turns $\Gamma(\mathcal{Q},\mathcal{D}_{\mathcal{Q}}^{\lambda})$ into a Dixmier algebra for $\mathbf{G}$. We can consider the category $M(\Gamma(\mathcal{D}_{\mathcal{Q}}^{\lambda}),\mathbf{K})$ of $\mathbf{K}$-equivariant modules for the Dixmier algebra $\Gamma(\mathcal{Q},\mathcal{D}_{\mathcal{Q}}^{\lambda})$ (cf Definition \ref{def:dixmiermodule}). Taking global sections defines a left-exact functor
\begin{equation}\label{eqn:globalsections}
\Gamma': \mathcal{M}(\mathcal{D}_{\mathcal{Q}}^{\lambda},\mathbf{K}) \to M(\Gamma(\mathcal{Q},\mathcal{D}_{\mathcal{Q}}^{\lambda}),\mathbf{K})\end{equation}
Pulling back along $\phi$ defines an exact functor
\begin{equation}\label{eqn:pullbackfunctor}
\phi^* :M(\Gamma(\mathcal{Q},\mathcal{D}_{\mathcal{Q}}^{\lambda}),\mathbf{K}) \to M(\mathfrak{g},\mathbf{K})
\end{equation}
We reserve the symbol $\Gamma$ for the composition $\phi^* \circ \Gamma'$. A module in the image of this functor has infinitesimal character $\gamma_{\lambda - \rho_{\mathfrak{g}}}$. When $\lambda- \rho_{\mathfrak{g}}$ is strictly antidominant, the functor $\Gamma'$ is particularly well-behaved. The following theorem was proved in the special case $\mathfrak{q}=\mathfrak{b}$ by Beilinson and Bernstein (\cite{BeilinsonBernstein1981}) and in the general case by Kitchen (\cite{Kitchen2011})

\begin{theorem}[\cite{BeilinsonBernstein1981},\cite{Kitchen2011}]
Suppose $\lambda -\rho_{\mathfrak{g}}$ is strictly antidominant, i.e.
\begin{equation}\label{eqn:weaklyantidominant}
\langle \lambda -\rho_{\mathfrak{g}}, \alpha^{\vee}\rangle < 0 \qquad \forall \alpha \in \Delta^+(\mathfrak{g},\mathfrak{h})\end{equation}
Then the functor
$$\Gamma': \mathcal{M}(\mathcal{D}_{\mathcal{Q}}^{\lambda},\mathbf{K}) \to M(\Gamma(\mathcal{Q},\mathcal{D}_{\mathcal{Q}}^{\lambda}),\mathbf{K})$$
is an exact equivalence of categories.
\end{theorem}

The work of Beilinson and Bernstein also leads to a classification of the irreducible objects in $M(\mathcal{D}_{\mathcal{Q}}^{\lambda},\mathbf{K})$ (without restrictions on $\lambda$). 

A \emph{standard datum} for ($\mathfrak{q}$, $\lambda$) is a pair $(Z, \tau)$ consisting of a $\mathbf{K}$-orbit $Z \subset \mathcal{Q}$ and a $\mathbf{K}$-equivariant line bundle $\tau$ on $Z$ which is suitably compatible with $\lambda$. To describe this compatibility, choose a point $\mathfrak{q}' \in Z$. The fiber of $\tau$ over $\mathfrak{q}'$ is a character $\tau_{\mathfrak{q}'}$ of $\mathbf{Q}'\cap \mathbf{K}$. To be compatible with $\lambda$, the differential of this character should coincide with the character of $\mathfrak{l}'$ determined by $\lambda$. This condition on $d\tau_{\mathfrak{q}'}$ is precisely what is required to define on $\tau$ the structure of a $\mathbf{K}$-equivariant $\mathcal{D}^{\lambda}_Z$-module.

If $\mathcal{M}$ is an irreducible $(\mathcal{M}^{\lambda}_{\mathcal{Q}},\mathbf{K})$-module, then $\mathrm{Supp}(\mathcal{M})$ is an irreducible subset of $\mathcal{Q}$ and, consequently, the closure in $\mathcal{Q}$ of a unique $\mathbf{K}$-orbit $Z$. Let $\mathcal{Q}' = \mathcal{Q} \setminus \partial Z$. Write
$$j: \mathcal{Q}' \hookrightarrow \mathcal{Q} \qquad i: Z \hookrightarrow \mathcal{Q}'$$
for the open and closed embeddings and let $k = j \circ i$. By Kashiwara's equivalence, there is a unique $(\mathcal{D}^{\lambda}_Z,\mathbf{K})$-module $\tau = k^!\mathcal{M}$ such that $i_!\tau = j^!\mathcal{M}$. On the other hand, $k_!\tau$ contains a unique irreducible $(\mathcal{D}_{\mathcal{Q}}^{\lambda},\mathbf{K})$-submodule, which is isomorphic to $\mathcal{M}$. This proves

\begin{theorem}\label{thm:bbequivalence1}
The map
$$\mathcal{M} \mapsto (Z,\tau)$$
characterized by the requirements $\mathrm{Supp}(\mathcal{M}) = \overline{Z}$ and $\tau = k^!\mathcal{M}$ defines a bijective correspondence between irreducible $(\mathcal{D}_{\mathcal{Q}}^{\lambda},\mathbf{K})$-modules and standard data for ($\mathfrak{q}$, $\lambda$). The inverse map takes the standard datum $(Z,\tau)$ to the unique irreducible $(\mathcal{D}_{\mathcal{Q}}^{\lambda},\mathbf{K})$-submodule of $k_!\tau$.
\end{theorem}

To elaborate on the relationship between irreducibles in $\mathcal{M}(\mathcal{D}_{\mathcal{Q}}^{\lambda},\mathbf{K})$ and standard data for $(\mathfrak{q},\lambda)$, we will pass to the Grothendieck group $K\mathcal{M}(\mathcal{D}_{\mathcal{Q}}^{\lambda},\mathbf{K})$. For every $p \geq 0$, we can define the left-derived functor
$$L^pk_!: \mathcal{M}(\mathcal{D}_Z^{\lambda},\mathbf{K}) \to \mathcal{M}(\mathcal{D}_{\mathcal{Q}}^{\lambda},\mathbf{K}) $$
Furthermore, $L^pk_!\tau = 0$ for $p >>0$. Hence, we can define
$$I(Z,\tau) := \sum_p (-1)^p [L^pk_!\tau] \in K\mathcal{M}(\mathcal{D}_{\mathcal{Q}}^{\lambda},\mathbf{K})$$
We will call this class the standard $(\mathcal{D}_{\mathcal{Q}}^{\lambda},\mathbf{K})$-module associated to $(Z,\tau)$. 

\begin{theorem}\label{thm:bbequivalence2}
The irreducible and standard $(\mathcal{D}_{\mathcal{Q}}^{\lambda},\mathbf{K})$-modules are related in $K\mathcal{M}(\mathcal{D}_{\mathcal{Q}}^{\lambda},\mathbf{K})$ by an upper triangular matrix with $\pm 1$'s along the diagonal. More precisely:

\begin{enumerate}
    \item Both sets form $\mathbb{Z}$-bases for $K\mathcal{M}(\mathcal{D}_{\mathcal{Q}}^{\lambda},\mathbf{K})$
    \item There is an ordering of the irreducible modules $\mathcal{M}_1,...,\mathcal{M}_n$ so that
    \begin{equation}\label{eqn:uppertriangularity}[\mathcal{M}_p] = \pm I(Z_p,\tau_p) + \sum_{q >p} c_{pq} I(Z_q,\tau_q) \qquad c_{pq} \in \mathbb{Z}\end{equation}
    where $(Z_p,\tau_p)$ is the standard datum associated to $\mathcal{M}_p$ under the bijection described in Theorem \ref{thm:bbequivalence1}.
\end{enumerate}
\end{theorem}

\begin{proof}
Choose an ordering $(Z_1,\tau_1), ..., (Z_n,\tau_n)$ of the standard data for $(\mathfrak{q},\lambda)$ so that
$$(Z_k,\tau_k) \leq (Z_l,\tau_l) \Leftrightarrow \dim(Z_k) \geq \dim(Z_l)$$
and induce an ordering $\mathcal{M}_1,...,\mathcal{M}_n$ on the irreducible objects using the bijection of Theorem \ref{thm:bbequivalence1}. 

Since $\mathbf{K}$ acts on $\mathcal{Q}$ with finitely many orbits, the $\mathbf{K}$-orbits of minimal dimension are necessarily closed. If $(Z_p,\tau_p)$ is a standard datum and $Z_p$ is closed, then by Kashiwara's theorem $I(Z_p,\tau_p) = [k_!\tau_p]$ is the class of an irreducible $(\mathcal{D}_{\mathcal{Q}}^{\lambda},\mathbf{K})$-module. In particular
$$[\mathcal{M}_p] = I(Z_p,\tau_p)$$
as desired.

Now let $(Z_p,\tau_p)$ be a standard datum for $(\mathfrak{q},\lambda)$ with the property that Equation \ref{eqn:uppertriangularity} holds for every $(Z_q,\tau_q)$ with $\dim(Z_q) < \dim(Z_p)$. We will show that Equation \ref{eqn:uppertriangularity} holds for $(Z_p,\tau_p)$. By Kashiwara's theorem, we can assume that $Z_p$ is an open subset of $\mathcal{Q}$. Write $l: \partial Z \hookrightarrow \mathcal{Q}$ for the closed embedding (recall that $i=j=k: Z \hookrightarrow \mathcal{Q}$ is the open embedding).

There is a distinguished triangle 
$$k_!k^!\mathcal{M}_p \to \mathcal{M}_p \to l_*l^*\mathcal{M}_p$$
in the bounded derived category $D^b\mathcal{M}(\mathcal{D}_{\mathcal{Q}}^{\lambda},\mathbf{K})$.

Taking Euler characteristics, we obtain an equation in $K\mathcal{M}(\mathcal{D}_{\mathcal{Q}}^{\lambda},\mathbf{K})$
$$[\mathcal{M}_p] = \chi k_!k^! \mathcal{M}_p + \chi l_*l^*\mathcal{M}_p$$
By definition, $\chi k_!k^!\mathcal{M}_p = I(Z_p,\tau_p)$. Furthermore
$$k^*l_*l^*\mathcal{M}_k = k^!l_*l^*\mathcal{M}_k = 0$$
Hence, $\chi l_*l^*\mathcal{M}_p$ is contained in $K\mathcal{M}_{\partial Z_p}(\mathcal{D}_{\mathcal{Q}}^{\lambda},\mathbf{K})$, and is therefore a linear combination of classes $I(Z_q,\tau_q)$ with $q > p$. This proves $(2)$. $(1)$ follows immediately. 
\end{proof}

Even when $\Gamma'$ fails to be exact, there is a well-defined class
$$\Gamma' I(Z,\tau):= \sum_p (-1)^p [\Gamma L^pk_!\tau]   \in KM^{\mathrm{fl}}(\Gamma(\mathcal{Q},\mathcal{D}_{\mathcal{Q}}^{\lambda}),\mathbf{K})$$

Up to a natural duality, this is exactly the class defined in Section \ref{sec:parabolicinduction}.

\begin{theorem}\label{thm:hmswduality}
Let $(Z,\tau)$ be a standard datum for $(\mathfrak{q},\lambda)$. Choose a parabolic subalgebra $\mathfrak{q}' \in Z$ and a Levi decomposition $\mathfrak{q}' = \mathfrak{l}' \oplus \mathfrak{u}'$. There is an equality in $KM^{\mathrm{fl}}(\Gamma(\mathcal{Q},\mathcal{D}_{\mathcal{Q}}^{\lambda}),\mathbf{K})$
$$\Gamma I(Z,\tau) = \pm I(\mathfrak{l}',\mathfrak{q}', \tau_{\mathfrak{q}'}^{\vee})^{\vee}$$
\end{theorem}

In the special case $\mathfrak{q}=\mathfrak{b}$, this is a direct consequence of the duality theorem of Hecht, Milicic, Schmid, and Wolf (\cite{HechtMilicicSchmidWolf}). The general case was handled by Kitchen (Theorem 1.2, \cite{Kitchen2011}).

\subsection{Translation Functors}\label{sec:translationfunctors}

In this section, we will define translation functors in a very general setting. This approach will illuminate some general properties of these functors which we will need in Section \ref{sec:inducedmainresults}.

\subsection{Peirce Decompositions}\label{sec:peirce}

Let $A$ be an associative ring with unit. An idempotent $e \in A$ is any element satisfying $e^2=e$. Two idempotents $e,e' \in A$ are orthogonal if $ee' = e'e = 0$. Suppose $\{e_i\}_{i=1}^n \subset A$ is a collection of mutually orthogonal idempotents and assume
$$\sum_{i=1}^n e_i = 1$$
If we write $A_{ij}$ for the subgroup $e_iAe_j \subset A$, there is a decomposition
$$A = 1A1 = \left(\sum_{i=1}^ne_i\right)A\left(\sum_{i=1}^ne_i\right) = \sum_{i,j} A_{ij}$$
which is direct by the orthogonality condition on $\{e_i\}$. This is called the Perice decomposition of $A$ associated to the set $\{e_i\}$. 

The pieces $A_{ij}$ satisfy the obvious relations
$$A_{ij}A_{kl} \subseteq \delta_{jk} A_{il}$$
In particular, each $A_{ii}$ is a ring (although not a subring of $A$. Its unit $e_i$ is \emph{not} the unit of $A$). If $M$ is an $A$-module (all modules will be left modules unless otherwise noted), then $e_iM$ is an $A_{ii}$-module and there is a decomposition
\begin{equation}\label{eqn:decompofM}
M = \bigoplus_i e_iM
\end{equation}
The assignment $M \mapsto e_iM$ defines a covariant functor
$$P_i: A-\mathrm{mod} \to A_{ii}-\mathrm{mod} \qquad P_iM = e_iM$$
which is exact by \ref{eqn:decompofM}.

Define a functor in the opposite direction
$$Q_i: A_{ii}-\mathrm{mod} \to A-\mathrm{mod} \qquad Q_iN = A \otimes_{A_{ii}} N$$

\begin{proposition}\label{prop:QiPi}
In the setting described above, $Q_i$ is right inverse and right adjoint to $P_i$.
\end{proposition}

\begin{proof}
Note that
$$Q_iN = A \otimes_{A_{ii}} N = Ae_i \otimes_{A_{ii}}N $$
Hence,
$$P_iQ_iN = e_i(Ae_i \otimes_{A_{ii}} N) = A_{ii} \otimes_{A_{ii}} N = N$$
This proves that $Q_i$ is right inverse to $P_i$. 

The functor $P_i$ is naturally equivalent to $\mathrm{Hom}_A(Ae_i, \cdot)$. The equivalence is implemented by the natural transformation
$$\eta_M: \mathrm{Hom}_A(Ae_i, M) \to e_iM  \qquad \eta_M(\varphi) = \varphi(e_i)$$
with inverse
$$\eta_M^{-1}: e_iM \to \mathrm{Hom}_A(Ae_i,M) \qquad \eta_M^{-1}(e_im)(ae_i) = ae_im$$
Hence the asserted adjunction is a special case of the the usual tensor-hom adjunction (for the $A-A_{ii}$-bimodule $Ae_i$).
\end{proof}

We will also need

\begin{proposition}\label{prop:uniquesimplequotient}
In the setting of Proposition \ref{prop:QiPi}, let $N$ be a nonzero simple $A_{ii}$-module. Then the $A$-module $Q_iN$ has a unique simple quotient, $\tilde{Q}_iN$, and there is an isomorphism of $A_{ii}$-modules $P_i\tilde{Q}_iN \cong N$.
\end{proposition}

\begin{proof}
Left multiplication by $e_i$ defines an $A_{ii}$-linear endomorphism $s_i$ of $Q_iN = A \otimes_{A_{ii}} N$ with image equal to $N$. Let $L$ be a proper $A$-submodule of $Q_iN$, and suppose $s_i(L) \neq 0$. Then by the simplicity of $N$, $s_i(L) = N$ and therefore $N \subseteq L$. But $N$ generates $Q_iN$ as an $A$-module, so in fact $Q_iN \subseteq L$ which contradicts the properness of $L \subset Q_iN$. Hence, $L \subseteq \ker{s_i}$.

Let $J \subset Q_iN$ be the sum of all proper submodules of $Q_iN$. By the argument above, $J \subseteq \ker{s_i}$. In particular, $J$ is the unique maximal proper submodule of $Q_iN$ and $e_iJ = 0$. Let
$$\tilde{Q}_iN := Q_iN/J$$
Then $\tilde{Q}_iN$ is the unique simple quotient of $Q_iN$ and
$$P_i\tilde{Q}_iN = e_i(Q_iN/J) \cong e_i(Q_i) = N$$
as desired. 
\end{proof}

\subsection{Translation Functors: Abstract Definition}\label{sec:abstracttranslation}

Now let $B$ be a (complex) associative algebra and let $F$ be a finite-dimensional vector space. Let $A = B \otimes_{\mathbb{C}} \mathrm{End}(F)$ and choose a family $\{e_i\} \subset A$ of mutually orthogonal idempotents with sum equal to $1$. Write
$$A = \bigoplus_{ij} A_{ij}$$
for the associated Peirce decomposition of $A$.

If $M$ is a $B$-module, then $M \otimes_{\mathbb{C}} F$ is an $A$-module. The assignment $M \mapsto M \otimes_{\mathbb{C}}F$ defines a covariant functor
$$t_F: B-\mathrm{mod} \to A-\mathrm{mod} \qquad t_F(M) = M \otimes_{\mathbb{C}}F$$
It is a standard fact that $t_F$ is an (exact) equivalence of categories, although its inverse is a bit harder to describe. 

Choose an idempotent $e_i \in A$. Let $T_{F,i}B := A_{ii}$. The abstract translation functor associated to the data of $B,F$, and $e_i$, is the composite
$$T_{F,i} := P_i \circ t_F: B-\mathrm{mod} \to T_{F,i}B-\mathrm{mod}$$
Since both $P_i$ and $t_F$ are exact, so is $T_{F,i}$. If $N$ is a nonzero simple $T_{F,i}B$-module, there is a canonically defined $B$-module $t_F^{-1}\tilde{Q}_iN$. By Proposition \ref{prop:uniquesimplequotient}, this $B$-module is simple and maps to $N$ under $T_{F,i}$.

\subsection{Translation Functors: Dixmier Algebras}\label{sec:translationdixmier}

We will apply the elementary results of Sections \ref{sec:peirce} and \ref{sec:abstracttranslation} in the following special case:

\begin{itemize}
    \item $I \subset U(\mathfrak{g})$ is a two-sided ideal such that $I \cap Z(\mathfrak{g})$ has finite codimension in $Z(\mathfrak{g})$
    \item $B = U(\mathfrak{g})/I$
    \item $F$ is a finite-dimensional representation of $\mathbf{G}$
    \item $A = B \otimes \mathrm{End}(F)$
\end{itemize}

By Propositions \ref{prop:dixmier1} and \ref{prop:dixmier2}, both $B$ and $A$ are Dixmier algebras. Condition \ref{cond:dixmier3} of Definition \ref{def:dixmier} implies that $A$ is $Z(\mathfrak{g})\times Z(\mathfrak{g})$-finite. Hence, $A$ admits a Peirce decomposition by generalized left and right infinitesimal characters:
$$A = \bigoplus_{\mu\nu} A_{\mu \nu} \qquad  A_{\mu \nu} := \{a \in A: \mathfrak{m}_{\mu}^pa = a\mathfrak{m}_{\nu}^q = 0 \text{ for } p,q >>0\}$$
where $\mathfrak{m}_{\mu}$ is the kernel in $Z(\mathfrak{g})$ of the infinitesimal character $\gamma_{\mu}$. If we fix $\mu$ appearing in the decomposition above, we get (as in Section \ref{sec:abstracttranslation}) an algebra $T_{F,\mu}B = A_{\mu\mu}$ (which is in fact a Dixmier algebra, since $B$ is) and a functor
$$T_{F,\mu}: B-\mathrm{mod} \to F_{F,\mu}B-\mathrm{mod}$$
As in Section \ref{sec:abstracttranslation}, $T_{F,\mu}$ is an exact functor and every nonzero simple $T_{F,\mu}B$-module has a canonically defined simple preimage.

If we introduce a $\mathbf{K}$-action, almost everything remains true. $T_{F,\mu}$ defines a functor
$$T_{F,\mu}: M(B,\mathbf{K}) \to M(T_{F,\mu}B,\mathbf{K})$$
Like its non-equivariant counterpart, this functor is exact and every nonzero simple $(T_{F,\mu}B,\mathbf{K})$-module has a canonical simple preimage.

\subsection{Main Results}\label{sec:inducedmainresults}

In this section, we analyze the effect of the global sections functor $\Gamma: \mathcal{M}(\mathcal{D}_{\mathcal{Q}}^{\lambda},\mathbf{K}) \to M(\mathfrak{g},\mathbf{K})$ on standard modules $I(Z,\tau)$ in the case of unipotent infinitesimal character. To simplify the arguments, we will assume that $\mathcal{O}_{\mathfrak{g}}$ is induced from $\{0\} \subset \mathfrak{l}$. The general case is proved by a similar argument. 

Our first observation is the following

\begin{proposition}\label{prop:phiforunipotentlambda}
Let $\mathcal{O}_{\mathfrak{g}} = \mathrm{Ind}^{\mathfrak{g}}_{\mathfrak{l}} \{0\}$ and suppose

\begin{enumerate}
    \item $\gamma_{\lambda -\rho_{\mathfrak{g}}} \in \unip(\mathcal{O}_{\mathfrak{g}})$
    \item The moment map
    $$\eta: T^*\mathcal{Q} \to \overline{\mathcal{O}}_{\mathfrak{g}}$$
    is birational
\end{enumerate}
Then the map
$$\phi: U(\mathfrak{g}) \to \Gamma(\mathcal{D}_{\mathcal{Q}}^{\lambda},\mathbf{K})$$
defined in Section \ref{sec:Dmodules} has the following properties:

\begin{enumerate}
    \item $\ker{\phi} = I_{\lambda -\rho_{\mathfrak{g}}}$ (cf Section \ref{sec:unipotentrepresentations}).
    \item $\phi$ is surjective
\end{enumerate}
\end{proposition}

\begin{proof}
Most of the necessary ingredients are contained in \cite{BorhoBrylinski1982}. There is a degree filtration on the Dixmier algebra $\Gamma(\mathcal{Q},\mathcal{D}_{\mathcal{Q}}^{\lambda})$. With respect to this filtation, there is a natural isomorphism of graded commutative algebras
$$\gr \Gamma(\mathcal{Q},\mathcal{D}_{\mathcal{Q}}^{\lambda}) \cong \Gamma(T^*\mathcal{Q},\mathcal{O}_{T^*\mathcal{Q}})$$
The algebra on the right is an integral domain (since the scheme $T^*\mathcal{Q}$ is reduced and irreducible). It follows for general reasons that $\Gamma(\mathcal{Q},\mathcal{D}_{\mathcal{Q}}^{\lambda})$ is zero-divisor free. Since $U(\mathfrak{g})/\ker{\phi}$ is included in this algebra, $\ker{\phi}$ is a completely prime ideal. Using Theorem \ref{thm:primitiveideals}, we deduce that
$$\ker{\phi} \in \mathrm{Prim}^{\lambda-\rho_{\mathfrak{g}}}U(\mathfrak{g})$$
We also have
$$\AV(\ker{\phi}) = \eta(T^*\mathcal{Q}) = \overline{\mathcal{O}}_{\mathfrak{g}}$$
by, e.g. Proposition 4.3 in \cite{BorhoBrylinski1982}. Hence, $\ker{\phi} = I_{\lambda -\rho_{\mathfrak{g}}}$ by Proposition \ref{prop:propsofunipotentideals}. 

Let $J$ be annihilator in  $U(\mathfrak{g})$ of the cokernel of $\phi$. Under the birationality assumption, there is a strict inclusion
$$\AV(J) \subset \eta(T^*\mathcal{Q}) = \overline{\mathcal{O}}_{\mathfrak{g}}$$
by, e.g. Corollary 5.12 in \cite{BorhoBrylinski1982}.

Since $J$ has infinitesimal character $\gamma_{\lambda-\rho_{\mathfrak{g}}}$, this implies that $J = U(\mathfrak{g})$ and hence that $\phi$ is surjective.
\end{proof}

Recall the factorization $\Gamma = \phi^* \circ \Gamma'$. Proposition \ref{prop:phiforunipotentlambda} tells us that $\phi^*$ is easy to understand when $\gamma_{\lambda-\rho_{\mathfrak{g}}}$ is unipotent. Next, we turn our attention to $\Gamma'$. If $\lambda-\rho_{\mathfrak{g}}$ is strictly antidominant, then $\Gamma'$ is an equivalence by Theorem \ref{thm:bbequivalence1}. Unfortunately, this condition is rarely satisfied when $\gamma_{\lambda-\rho_{\mathfrak{g}}}$ is unipotent. To apply Theorems \ref{thm:bbequivalence1} and \ref{thm:bbequivalence2}, we will first need to translate to a more regular infinitesimal character. The following proposition is critical

\begin{proposition}[\cite{Vogan1988b}, Proposition 4.7]\label{prop:translationweaklyfair1}
Suppose $\lambda-\rho(\mathfrak{u})$ is antidominant for $\mathfrak{u}$, i.e.
$$\langle \lambda - \rho(\mathfrak{u}),\alpha^{\vee}\rangle \leq 0  \qquad \alpha \in \Delta^+(\mathfrak{u},\mathfrak{h})$$
Let $\xi \in \mathfrak{h}^*$ be the weight of a one-dimensional representation $\mathbb{C}_{\xi}$ of $\mathbf{Q}$, and let $F_{\xi}$ be the irreducible representation of $\mathbf{G}$ of highest weight $\xi$. Then there is a natural isomorphism of Dixmier algebras
$$\Gamma(\mathcal{Q},\mathcal{D}_{\mathcal{Q}}^{\lambda}) \cong T_{F_{\xi},\lambda-\rho_{\mathfrak{g}}}\Gamma(\mathcal{Q},\mathcal{D}_{\mathcal{Q}}^{\lambda-\xi})$$
\end{proposition}

In the setting of Proposition \ref{prop:translationweaklyfair1}, $T_{F_{\xi},\lambda-\rho_{\mathfrak{g}}}$ defines a surjective group homomorphism
$$T_{F_{\xi},\lambda-\rho_{\mathfrak{g}}}: KM(\Gamma(\mathcal{Q},\mathcal{D}_{\mathcal{Q}}^{\lambda-\xi}),\mathbf{K}) \twoheadrightarrow KM(\Gamma(\mathcal{Q},\mathcal{D}_{\mathcal{Q}}^{\lambda}),\mathbf{K})$$
Write $\mathcal{L}^{\xi} \to \mathcal{Q}$ for the $\mathbf{G}$-equivariant line bundle corresponding to $\mathbb{C}_{\xi}$. If $(Z,\tau)$ is a standard datum for $(\mathfrak{q},\lambda-\xi)$, then $(Z,\tau \otimes \mathcal{L}^{\xi}|_Z)$ is a standard datum for $(\mathfrak{q},\lambda)$. This defines a bijection between standard data for $(\mathfrak{q},\lambda-\xi)$ and standard data for $(\mathfrak{q},\lambda)$ (with inverse $(Z,\tau) \mapsto (Z, \tau \otimes \mathcal{L}^{-\xi}|_Z)$). By the proof of Proposition \ref{prop:phiforunipotentlambda} in \cite{Vogan1988b}

\begin{proposition}[\cite{Vogan1988b}, Proposition 4.7]\label{prop:translationweaklyfair2}
In the setting of Proposition \ref{prop:translationweaklyfair1}, let $(Z,\tau)$ be a standard datum for $(\mathfrak{q},\lambda)$. Then there is an equality in $KM(\Gamma(\mathcal{Q},\mathcal{D}_{\mathcal{Q}}^{\lambda},\mathbf{K})$
$$\Gamma' I(Z,\tau) = T_{F_{\xi},\lambda-\rho_{\mathfrak{g}}} \Gamma' I(Z,\tau \otimes \mathcal{L}^{-\xi}|_Z)$$
\end{proposition}

We deduce

\begin{theorem}\label{thm:inductionmainresult1}
Let $\mathcal{O}_{\mathfrak{g}} = \mathrm{Ind}^{\mathfrak{g}}_{\mathfrak{l}} \{0\}$, and suppose

\begin{enumerate}
    \item $\gamma_{\lambda - \rho_{\mathfrak{g}}} \in \unip(\mathcal{O}_{\mathfrak{g}})$
    \item $\lambda - \rho(\mathfrak{u})$ is antidominant for $\mathfrak{u}$
    \item The moment map
    $$\eta: T^*\mathcal{Q} \to \overline{\mathcal{O}}_{\mathfrak{g}}$$
    is birational
\end{enumerate}
Then the sets $\Unip^{\lambda-\rho_{\mathfrak{g}}}(\mathcal{O}_{\mathfrak{g}})$ and $\mathrm{Deg}^{\lambda + \rho_{\mathfrak{l}}}(\mathfrak{q})$ are related by an upper triangular matrix. More precisely
\begin{enumerate}
    \item $\Unip^{\lambda-\rho_{\mathfrak{g}}}(\mathcal{O}_{\mathfrak{g}})$ and $\mathrm{Deg}^{\lambda + \rho_{\mathfrak{l}}}(\mathfrak{q})$ have the same $\mathbb{Z}$-span in $KM^{\mathrm{fl}}(\mathfrak{g},\mathbf{K})$
    \item There is a natural injection
    $$\Unip^{\lambda-\rho_{\mathfrak{g}}}(\mathcal{O}_{\mathfrak{g}}) \hookrightarrow \mathrm{Deg}^{\lambda + \rho_{\mathfrak{l}}}(\mathfrak{q})$$
    and an ordering of the unipotent modules $U_1,...,U_n$ such that
    $$[U_p] = \pm [D_p] + \sum_{q >p} c_{pq} [D_q] \qquad c_{pq} \in \mathbb{Z}$$
    for an ordering $D_1,...,D_m$ of the degenerate modules compatible with this injection.
\end{enumerate}
\end{theorem}

\begin{proof}
Suppose $[D] \in \Deg^{\lambda+\rho_{\mathfrak{l}}}(\mathfrak{q})$. By Proposition \ref{prop:phiforunipotentlambda} and the definition of $\Deg^{\lambda+\rho_{\mathfrak{l}}}(\mathfrak{q})$, $D$ is contained in the subgroup $KM^{\mathrm{fl}}(U(\mathfrak{g})/I_{\lambda-\rho_{\mathfrak{g}}},\mathbf{K}) \subset KM^{\mathrm{fl}}(\mathfrak{g},\mathbf{K})$. Hence, $[D]$ is a linear combination of $\Unip^{\lambda-\rho_{\mathfrak{g}}}(\mathcal{O}_{\mathfrak{g}})$.

Conversely, suppose $[U] \in \Unip^{\lambda-\rho_{\mathfrak{g}}}(\mathcal{O}_{\mathfrak{g}})$. For every $\alpha \in \Delta(\mathfrak{u},\mathfrak{h})$, there is a strict inequality
$$\langle \rho(\mathfrak{u}),\alpha^{\vee}\rangle = \langle \rho_{\mathfrak{g}}, \alpha^{\vee}\rangle - \langle \rho_{\mathfrak{l;}}, \alpha^{\vee}\rangle = \langle \rho_{\mathfrak{g}},\alpha^{\vee}\rangle > 0$$
Choose $N >>0$ so that 
$$\label{eqn:bigN}\langle \lambda - 2N\rho(\mathfrak{u}) - \rho_{\mathfrak{g}}\rangle < 0 \qquad \forall \alpha \in \Delta(\mathfrak{u},\mathfrak{h})$$
Then if $\xi = 2N\rho(\mathfrak{u})$, $
\lambda - \xi$ satisfies the antidominance condition of Proposition \ref{thm:bbequivalence1}. Note that $\xi$ is the weight of the character $\left(\wedge^{\mathrm{top}}(\mathfrak{u})\right)^N$ of $\mathbf{Q}$.

By Proposition \ref{prop:phiforunipotentlambda}, there is a unique irreducible $(\Gamma(\mathcal{Q},\mathcal{D}_{\mathcal{Q}}^{\lambda}),\mathbf{K})$-module $U'$ such that $U = \phi^*U'$. By the results of Section \ref{sec:translationdixmier}, there is a canonically-defined irreducible $(\Gamma(\mathcal{Q},\mathcal{D}_{\mathcal{Q}}^{\lambda}),\mathbf{K})$-module $\tilde{U}'$ such that $U' = T_{F_{\xi},\lambda-\rho_{\mathfrak{g}}}\tilde{U}'$. And by Theorem \ref{thm:bbequivalence1}, there is a unique irreducible $(\mathcal{D}_{\lambda-\xi},\mathbf{K})$-module $\tilde{\mathcal{U}}'$ such that $\tilde{U}' = \Gamma'\tilde{\mathcal{U}}'$. The assignment $[U] \mapsto \tilde{\mathcal{U}}'$ defines an injection 
$$\Unip^{\lambda-\rho_{\mathfrak{g}}}(\mathcal{O}_{\mathfrak{g}}) \hookrightarrow \mathrm{Irr}(\mathcal{D}_{\mathcal{Q}^{\lambda-\xi}},\mathbf{K})$$
By Theorem \ref{thm:bbequivalence2}, there is an ordering $\mathcal{M}_1,...\mathcal{M}_n$ of $\mathrm{Irr}(\mathcal{D}_{\mathcal{Q}^{\lambda-\xi}},\mathbf{K})$ such that

$$
[\mathcal{M}_p] = \pm I(Z_p,\tau_p) + \sum_{q >p} c_{pq} I(Z_q,\tau_q) \qquad c_{pq} \in \mathbb{Z}$$
where $(Z_1,\tau_1),...,(Z_n,\tau_n)$ is the induced ordering on the standard data for $(\mathfrak{q},\lambda-\xi)$. Define an ordering $U_1,...,U_r$ on $\Unip^{\lambda-\rho_{\mathfrak{g}}}(\mathcal{O}_{\mathfrak{g}})$ using the injection described above. Then there is an ordering on the standard data $(Z_i, \tau_i)$ so that
\begin{equation}\label{eqn:uppertriangularity2}[\tilde{\mathcal{U}}'_p] = \pm I(Z_p,\tau_p) + \sum_{q >p} c_{pq} I(Z_q,\tau_q) \qquad c_{pq} \in \mathbb{Z}\end{equation}
The composite $\phi^* \circ T_{F_{\xi},\lambda-\rho_{\mathfrak{g}}} \circ \Gamma': \mathcal{M}(\mathcal{D}_{\mathcal{Q}}^{\lambda-\xi},\mathbf{K}) \to M(\mathfrak{g},\mathbf{K})$ is exact, and therefore defines a group homomorphism $K\mathcal{M}(\mathcal{D}_{\mathcal{Q}}^{\lambda-\xi},\mathbf{K}) \to KM(\mathfrak{g},\mathbf{K})$. If we apply this homomorphism to both sides of \ref{eqn:uppertriangularity2} and use Proposition \ref{prop:translationweaklyfair2} to simplify, we obtain equalities in $KM(\mathfrak{g},\mathbf{K})$
$$[U_p] = \pm \Gamma I(Z_p,\tau_p \otimes \mathcal{L}^{\xi}|_Z) + \sum_{q >p} d_{pq} \Gamma I(Z_q,\tau_q \otimes \mathcal{L}^{\xi}|_Z) \qquad d_{pq} \in \mathbb{Z} $$
The classes appearing on the right hand side are the degenerate modules by Theorem \ref{thm:hmswduality}.
\end{proof}

A similar argument shows

\begin{theorem}
Let $\mathcal{O}_{\mathfrak{g}} = \mathrm{Ind}^{\mathfrak{g}}_{\mathfrak{l}}\mathcal{O}_{\mathfrak{l}}$ and let $\mu \in \mathfrak{h}^*$. Suppose

\begin{enumerate}
    \item $\gamma_{\mu} \in \unip(\mathcal{O}_{\mathfrak{l}})$
    \item $\gamma_{\mu - \rho(\mathfrak{u})} \in \unip(\mathcal{O}_{\mathfrak{g}})$
    \item $\mu - \rho(\mathfrak{u})$ is antidominant for $\mathfrak{u}$
    \item The moment map
    $$\eta: \mathbf{G} \times_{\mathbf{Q}} \left(\overline{\mathcal{O}}_{\mathfrak{l}} + \mathfrak{u}\right) \to \overline{\mathcal{O}}_{\mathfrak{g}}$$
    is birational
\end{enumerate}
Then the sets $\Unip^{\mu-\rho(\mathfrak{u})}(\mathcal{O}_{\mathfrak{g}})$ and $\mathrm{Deg}^{\mu}(\mathfrak{q})$ are related by an upper triangular matrix (in the sense of Theorem \ref{thm:inductionmainresult1}).
\end{theorem}

\section{Examples}\label{sec:examples}

Let $G = Sp(2n,\mathbb{R})$. The $\mathbf{G}$-orbits on $\mathcal{N}_{\mathfrak{g}}$ are parameterized by \emph{even} partitions of $2n$. Let $\mathcal{O}$ be the orbit corresponding to the partition $2+2+...+2$. With respect to the usual symplectic form on $\mathbb{C}^{2n}$
$$\mathfrak{g} = \left\{\left(
\begin{array}{c|c}
A & B\\ \hline
C & -A^t
\end{array}\right): B,C \text{ symmetric}\right\}$$
and $\mathcal{O}$ is the $\mathbf{G}$-orbit of the matrix
$$\left(
\begin{array}{c|c}
0 & I_n\\ \hline
0 & 0
\end{array}\right)$$
It is not hard to see that $\mathcal{O}$ is birationally induced from the $\{0\}$-orbit of the Segal parabolic
$$\mathfrak{l} = \left\{\left(
\begin{array}{c|c}
A & 0\\ \hline
0 & -A^t
\end{array}\right)\right\} \qquad \mathfrak{u} = \left\{\left(
\begin{array}{c|c}
0 & B\\ \hline
0 & 0
\end{array}\right): B \text{ symmetric}\right\} \qquad \mathfrak{q} = \mathfrak{l} \oplus \mathfrak{u}$$

Choose $\theta$ so that $\mathfrak{k} \subset \mathfrak{g}$ is the diagonal copy of $\mathfrak{gl}_2(\mathbb{C})$ (i.e. so that $\mathfrak{k} = \mathfrak{l}$). Then $\mathcal{O} \cap \mathfrak{p}$ decomposes into $n+1$ $\mathbf{K}$-orbits, which are parameterized by pairs $(a,b)$ of nonnegative integers with $a+b =n$. Write $\mathcal{O}_{(a,b)}$ for the nilpotent $\mathbf{K}$-orbit corresponding to the pair $(a,b)$. Then $\mathcal{O}_{(a,b)}$ is the $\mathbf{K}$-orbit of the matrix
$$e_{(a,b)}:=\left(
\begin{array}{c|c}
0 & \begin{array}{cc}I_a & 0\\0 & 0 \end{array}\\ \hline
\begin{array}{cc}0 & 0\\0 & I_b\end{array} & 0
\end{array}\right)$$
Note that the centralizer in $\mathbf{K}$ of $e_{a,b}$ is the subgroup
$$\mathbf{K}^{e_{(a,b)}} = \left\{\left( \begin{array}{c|c}O_a(\mathbb{C}) & \ast \\ \hline 0 & O_b(\mathbb{C}) \end{array}\right)\right\} \hookrightarrow Sp(2n,\mathbb{C})$$
Hence, if $a,b \geq 1$, there are four $\mathbf{K}$-equivariant line bundles on $\mathcal{O}_{a,b}$, parameterized by pairs $(x,y) \in \{0,1\}^2$. The line bundle $\mathcal{L}_{(x,y)}$ corresponding to $(x,y)$ has fiber over $e_{(a,b)}$ equal to the character $\det^x \otimes \det^y$ of $\mathbf{K}^{e_{(a,b)}}$. If either $a$ or $b$ is equal to $0$, then there are only two line bundles $\mathcal{L}_0$ and $\mathcal{L}_1$ corresponding to the trivial and determinant characters of $O_n(\mathbb{C})$.

Choose the standard (diagonal) Cartan subalgebra $\mathfrak{h} \subset \mathfrak{g}$ and the usual positive system $\Delta^+(\mathfrak{g},\mathfrak{h}) = \{e_i-e_j\}_{i <j} \cup \{e_i+e_j\}_{i<j} \cup \{2e_i\}$.

Compute
$$\rho_{\mathfrak{l}} = \frac{1}{2}(n-1,n-3,...,3-n,1-n) \qquad \rho(\mathfrak{u}) = \frac{1}{2}(n+1,...,n+1) \qquad \rho_{\mathfrak{g}} = (n,n-1,...,1)$$
Note that $\lambda = \rho(\mathfrak{u})$ is a one-dimensional representation of $\mathfrak{l}$ and satisfies the conditions of Theorem \ref{thm:inductionmainresult1}: $\gamma_{\lambda-\rho_{\mathfrak{g}}}= \gamma_{\rho_{\mathfrak{l}}}$ is a unipotent infinitesimal character associated to $\mathcal{O}$ (in fact, it is one of the special unipotent infinitesimal characters defined by Barbasch and Vogan in \cite{BarbaschVogan1985}), and $\lambda - \rho(\mathfrak{u}) = 0$ is antidominant for $\mathfrak{u}$. Hence, Theorem \ref{thm:inductionmainresult1} implies that the sets $\Unip^{\rho_{\mathfrak{l}}}(\mathcal{O})$ and $\Deg^{\rho_{\mathfrak{g}}}(\mathfrak{q})$ are related by an upper triangular matrix. Using the Atlas software, we can compute its entries explicitly. We do so below for $n=2,3$, and $5$. 

For each group, we have included two tables. The first table lists the elements of $\Unip^{\rho(\mathfrak{l})}(\mathcal{O})$. For each representation, we record the associated variety (i.e. the set of open $\mathbf{K}$-orbits $\mathcal{O}_{(a,b)}$ therein) and the associated vector bundles. 

The second table lists the degenerate representations. For each representation, we record the support $\overline{Z}$ of the corresponding standard $(\mathcal{D}^{\lambda}_{\mathcal{Q}},\mathbf{K})$-module (by indicating the open KGB element in the preimage of $Z$ in the full flag variety) and its decomposition into unipotent representations. The KGB element is indicated in the Atlas notation. Note that some degenerates are $0$. Others are duplicated (when they arise as the global sections of several different standards).

\subsection{$Sp(4,\mathbb{R})$}

\paragraph{Unipotents}

\begin{center}
\begin{tabular}{|c|c|c|}
\hline
& Associated Variety & Associated Vector Bundles \\ \hline
$1$ & $\mathcal{O}_{2,0},\mathcal{O}_{1,1}, \mathcal{O}_{0,2}$ & $\mathcal{L}_{(0)}, \mathcal{L}_{(0,0)}, \mathcal{L}_{(0)}$ \\ \hline
$2$ & $\mathcal{O}_{2,0},\mathcal{O}_{1,1}, \mathcal{O}_{0,2}$ & $\mathcal{L}_{(1)}, \mathcal{L}_{(1,1)}, \mathcal{L}_{(1)}$ \\ \hline
\end{tabular}
\end{center}

\paragraph{Degenerates}

\begin{center}
\begin{tabular}{|c|c|}
\hline
Support (KGB) & Decomposition into Unipotents\\ \hline
$10$ & $1$\\ \hline
$10$ & $2$ \\ \hline
\end{tabular}
\end{center}

\subsection{$Sp(6,\mathbb{R})$}

\paragraph{Unipotents}

\begin{center}
\begin{tabular}{|c|c|c|}
\hline
 & Associated Variety & Associated Vector Bundles \\ \hline
$1$ & $\mathcal{O}_{(3,0)}$ & $\mathcal{L}_{(0)}$\\ \hline
$2$ & $\mathcal{O}_{(2,1)}$ & $\mathcal{L}_{(1,0)}$ \\ \hline
$3$ & $\mathcal{O}_{(1,2)}$ & $\mathcal{L}_{(0,1)}$ \\ \hline
$4$ & $\mathcal{O}_{(0,3)}$ & $\mathcal{L}_{(0)}$ \\ \hline
$5$ & $\mathcal{O}_{(3,0)}, \mathcal{O}_{(2,1)}$ & $\mathcal{L}_{(1)}, \mathcal{L}_{(1,1)}$ \\ \hline
$6$ & $\mathcal{O}_{(2,1)},\mathcal{O}_{(1,2)}$ & $\mathcal{L}_{(0,0)},\mathcal{L}_{(0,0)}$ \\ \hline
$7$ & $\mathcal{O}_{(1,2)},\mathcal{O}_{(0,3)}$ & $\mathcal{L}_{(1,1)},\mathcal{L}_{(1)}$  \\ \hline
\end{tabular}
\end{center}

\paragraph{Degenerates}

\begin{center}
\begin{tabular}{|c|c|c|}
\hline
Support (KGB) & Decomposition into Unipotents \\ \hline
$7$ &  $1$\\ \hline
$18$ & $2$\\ \hline
$31$ &  $2+1$\\ \hline
$17$ & $3$\\ \hline
$5$ &  $4$\\ \hline
$30$ & $4+3$\\ \hline
$31$ &  $5$\\ \hline
$44$ & $6+4+1$\\ \hline
$34$ & $6-2-3$\\ \hline
$40$ &  $6+4-3$\\ \hline
$40$ &  $6+4-3$\\ \hline
$41$ &  $6+1-2$\\ \hline
$41$ &  $6+1-2$\\ \hline
$30$ & $7$\\ \hline
$44$ &  $7+5$\\ \hline
$34$ & $0$\\ \hline
\end{tabular}
\end{center}

\subsection{$Sp(10,\mathbb{R})$}

\begin{center}
\begin{tabular}{|c|c|c|}
\hline
 & Associated Varieties & Associated Vector Bundles \\ \hline
1 & $\mathcal{O}_{(5,0)}$ & $\mathcal{L}_{(1)}$  \\ \hline
2 & $\mathcal{O}_{(4,1)}$ & $\mathcal{L}_{(0,1)}$ \\ \hline
3 & $\mathcal{O}_{(3,2)}$ & $\mathcal{L}_{(1,0)}$ \\ \hline
4 & $\mathcal{O}_{(2,3)}$ & $\mathcal{L}_{(0,1)}$ \\ \hline
5 & $\mathcal{O}_{(1,4)}$ & $\mathcal{L}_{(1,0)}$ \\ \hline
6 & $\mathcal{O}_{(0,5)}$ & $\mathcal{L}_{(1)}$ \\ \hline
7 & $\mathcal{O}_{(5,0)},\mathcal{O}_{(4,1)}$ & $\mathcal{L}_{(0,0)},\mathcal{L}_{(0,0)}$\\ \hline
8 & $\mathcal{O}_{(4,1)},\mathcal{O}_{(3,2)}$ & $\mathcal{L}_{(1,1)},\mathcal{L}_{(1,1)}$ \\ \hline
9 & $\mathcal{O}_{(3,2)},\mathcal{O}_{(2,3)}$ & $\mathcal{L}_{(0,0)},\mathcal{L}_{(0,0)}$ \\ \hline
10 & $\mathcal{O}_{(2,3)},\mathcal{O}_{(1,4)}$ & $\mathcal{L}_{(1,1)},\mathcal{L}_{(1,1)}$\\ \hline
11 & $\mathcal{O}_{(1,4)}, \mathcal{O}_{(0,5)}$ & $\mathcal{L}_{(0,0)},\mathcal{L}_{(0,0)}$ \\ \hline
\end{tabular}
\end{center}

\paragraph{Degenerates}

\begin{center}
\begin{tabular}{|c|c|}
\hline
Support (KGB) & Decomposition Into Unipotents \\ \hline
$28$ & $1$\\ \hline
$304$ &  $2$ \\ \hline
$414$ &  $-2-1$\\ \hline
$469$ &  $3$\\ \hline
$469$ &  $3$\\ \hline
$779$ &  $-3-1-1$\\ \hline
$470$ &  $4$\\ \hline
$470$ &  $4$\\ \hline
$877$ &  $-4-2-2$\\ \hline
$740$ & $-4-3$\\ \hline
$306$ &  $5$\\ \hline
$878$ &  $-5-5-3$\\ \hline
$956$ &  $5+3+1$\\ \hline
$31$ & $6$\\ \hline
$781$ &  $-6-6-4$\\ \hline
$416$ &  $-6-5$\\ \hline
$957$ &  $6+4+2$\\ \hline
$414$ &  $7$\\ \hline
$690$ &  $8-3-2$\\ \hline
$779$ &  $-8+2-1-1$\\ \hline
$877$ &  $-8-4+3-2$\\ \hline
$921$ &  $8+4+2+1+1$\\ \hline
$740$ &  $-9$\\ \hline
$921$ &  $9+7+7$\\ \hline
$691$ &  $10-5-4$\\ \hline
$781$ &  $-10-6-6+5$\\ \hline
$878$ &  $-10-5+4-3$\\ \hline
$922$ &  $10+6+6+5+3$\\ \hline
$941$ &  $-10-8+5+4+3+2$\\ \hline
$956$ &  $10+8-4-2+1$\\ \hline
$957$ &  $10+8-3-5+6$\\ \hline
$962$ &  $-10-8-6-1$\\ \hline
$416$ &  $11$\\ \hline
$922$ &  $11+11+9$\\ \hline
$962$ &  $-11-9-7$\\ \hline
$690$ &  $0$\\ \hline
$691$ &  $0$\\ \hline
$941$ &  $0$\\ \hline
\end{tabular}
\end{center}
\bibliographystyle{plain}
\bibliography{bibliography.bib}

\end{document}